\documentclass[english]{amsart}

\RequirePackage{amsxtra}
\RequirePackage{amsfonts}
\RequirePackage{latexsym}

\RequirePackage{amsthm}
\usepackage{stmaryrd}
\usepackage{mathtools}

\usepackage{amsmath}

\usepackage{amssymb}

\usepackage{amscd} \usepackage{tabularx}
\usepackage{enumerate}
\usepackage[all,cmtip,line]{xy}

\newcommand{\tz}{\tilde{z}}
\newcommand{\GQp}{G_{\mathbb{Q}_p}}
\newcommand{\F}{\mathbb{F}} 
\renewcommand{\ltimes}{l^\times}
\newcommand{\ktimes}{k^\times}

\newcommand{\sigmabar}{\overline{\sigma}}

\def\RCS$#1: #2 ${\expandafter\def\csname RCS#1\endcsname{#2}}
\RCS$Revision: 191 $
\RCS$Date: 2011-06-20 19:15:39 -0500 (Mon, 20 Jun 2011) $

\DeclareMathOperator{\cind}{c-Ind}
\DeclareMathOperator{\ind}{Ind}
\DeclareMathOperator{\Nm}{Nm}
\def\G{\GL_2(K)}
\def\GO{\GL_2(\cO_K)}
\def\D{D^\times}
\def\OD{\cO_D^\times}
\def\CGL{\mathcal{C}(\GL_2(k))}
\def\Cl{\mathcal{C}(l^\times)}
\def\sp{\mathrm{sp}}
\def\crs{\mathrm{cr}}

\newcommand{\To}{\longrightarrow}\newcommand{\into}{\hookrightarrow}

\newcommand{\onto}{\twoheadrightarrow}
\newcommand{\isoto}{\stackrel{\sim}{\To}} 
 
 \newcommand{\rec}{\operatorname{rec}}

\newcommand{\Z}{\mathbb{Z}} 
\newcommand{\Q}{\mathbb{Q}}

\newcommand{\wt}{\widetilde}

\DeclareMathOperator{\red}{red}

\DeclareMathOperator{\sesi}{ss}

\DeclareMathOperator{\JL}{JL}

\newcommand{\rhobar}{\overline{\rho}}

\newcommand{\GL}{\operatorname{GL}}

\newcommand{\Sp}{\operatorname{Sp}}

\newcommand{\HT}{\operatorname{HT}}

\newcommand{\CO}{{\mathcal{O}}}

\newcommand{\cC}{\mathcal{C}}

\newcommand{\cO}{\mathcal{O}}

 \newcommand{\Qp}{\Q_p}
 
\newcommand{\Qpbar}{\overline{\Q}_p}
\newcommand{\Zpbar}{\overline{\Z}_p}

\newcommand{\Fpbar}{\overline{\F}_p}

\newcommand{\Fp}{{\F_p}}

\newcommand{\Res}{\operatorname{Res}} 

\newcommand{\Spec}{\operatorname{Spec}}
\newcommand{\Ind}{\operatorname{Ind}}

\newcommand{\Hom}{\operatorname{Hom}}

\newcommand{\Sym}{\operatorname{Sym}}

\newcommand{\diag}{\operatorname{diag}}

\newcommand{\tR}{\widetilde{R}}

\DeclareMathOperator{\ds}{ds}

\usepackage{amsthm}

\newtheorem{thm}[subsection]{Theorem}

\newtheorem{lem}[subsection]{Lemma}

\newtheorem{cor}[subsection]{Corollary}
\newtheorem{conj}[subsection]{Conjecture}

\newtheorem{prop}[subsection]{Proposition}

\theoremstyle{definition}

\newtheorem{defn}[subsection]{Definition}

\theoremstyle{remark}

\newtheorem{rem}[subsection]{Remark}

\usepackage{hyperref}
\begin{document}
\title {The Breuil--M\'ezard Conjecture for quaternion algebras}

\author{Toby Gee} \email{toby.gee@imperial.ac.uk} \address{Department of
  Mathematics, Imperial College London} \author{David Geraghty}
\email{geraghty@math.princeton.edu}\address{Princeton University and
  Institute for Advanced Study} 
\thanks{The first author was
  supported in part by a Marie Curie Career Integration Grant, and by an
  ERC Starting Grant. The second author was partially supported by NSF grants
  DMS-1200304 and DMS-1128155.}
\begin{abstract}We formulate a version of the Breuil--M\'ezard
  conjecture for quaternion algebras, and show that it follows from
  the Breuil--M\'ezard conjecture for $\GL_2$. In the course of the
  proof we establish a mod $p$ analogue of the Jacquet--Langlands
  correspondence for representations of $\GL_2(k)$, $k$ a finite field
  of characteristic $p$.
\end{abstract}
\maketitle
\tableofcontents
\section{Introduction.}\label{sec:intro}

The Breuil--M\'ezard conjecture (\cite{breuil-mezard}) has proved to
be one of the most important conjectures linking Galois
representations and automorphic forms; indeed, Kisin's proof of (most
cases of) the original formulation of the conjecture (\cite{kisinfmc}) simultaneously
established (most cases of) the Fontaine--Mazur conjecture for
$\GL_2/\Q$. The original conjecture predicted the Hilbert--Samuel
multiplicities of the special fibres of potentially semistable
deformation rings for two-dimensional mod $p$ representations of
$\GQp$, the absolute Galois group of $\Qp$, in terms of the
representation theory of $\GL_2(\Z_p)$. The statement of the conjecture was generalised
in \cite{kisinICM} to the case of representations of $G_K$, for $K$ an
arbitrary finite extension of $\Qp$. This conjecture is largely open,
although it has been proved for potentially Barsotti--Tate
representations (\cite{geekisin}).

The connection between potentially semistable deformation rings
and the representation theory of $\GL_2$ is via the local Langlands
correspondence. Given the Jacquet--Langlands correspondence, it is
natural to wonder whether for potentially semistable deformation rings
of discrete series type, the Hilbert--Samuel multiplicities could also
be described in terms of the representation theory of the units in a non-split
quaternion algebra. One advantage of such a description is that the
representation theory is much simpler in this case; all the
irreducible admissible representations are finite-dimensional, and the
irreducible mod $p$ representations of the maximal compact subgroup are all one-dimensional.

In this paper we formulate such a conjecture, and show that it is a
consequence of the conjecture for $\GL_2$. In particular, we prove the
conjecture (in most cases) over $\Qp$, as a consequence of Kisin's
proof of the conjecture for $\GL_2$ in this case (\cite{kisinfmc}). In
order to do this, we have found it helpful to reformulate the
conjecture slightly more abstractly in terms of linear functionals on
Grothendieck groups of representations, and also to prove a general
result on the reduction modulo $p$ of the Jacquet--Langlands
correspondence, or rather a version of this correspondence for
types. In the remainder of the introduction, we will explain this in
some detail.

Let $K$ be a finite extension of $\Qp$ with absolute Galois group
$G_K$, and let $\tau$ be an inertial type for $K$ (i.e.\ a
two-dimensional representation of the inertia group $I_K$ with open
kernel, which can be extended to $G_K$). Let $\lambda$ be a highest
weight for an irreducible algebraic representation of $\Res_{K/\Q_p}\GL_{2_{/K}}$. Then
a recipe using Henniart's inertial local Langlands correspondence (see
the appendix to \cite{breuil-mezard}) associates to the pair
$(\tau,\lambda)$ a finite-dimensional irreducible representation
$\sigma(\tau,\lambda)$ of $\GL_2(\cO_K)$ over $\Qpbar$. Let $k$ denote the residue
field of $\cO_K$. Choosing a stable lattice, reducing modulo $p$ and semisimplifying, we can
write \[\sigmabar(\tau,\lambda)\cong\oplus_{\sigmabar}\sigmabar^{n_{\tau,\lambda}(\sigmabar)},\]where
$\sigmabar$ runs over the equivalence classes of irreducible mod $p$
representations of $\GL_2(k)$, and $n_{\tau,\lambda}(\sigmabar)$ is a
nonnegative integer.

Let $\rhobar:G_K\to\GL_2(\Fpbar)$ be a continuous representation. Then
there is (after fixing a sufficiently large coefficient field) a
universal lifting ring $R^{\tau,\lambda}$ for lifts of $\rhobar$ which
are potentially semistable of inertial type $\tau$ and Hodge type
$\lambda$. Let $e(R^{\tau,\lambda}/\varpi)$ denote the Hilbert--Samuel
multiplicity of the special fibre of $R^{\tau,\lambda}$. Then the
Breuil--M\'ezard conjecture asserts that there are uniquely determined
nonnegative integers $\mu_{\sigmabar}(\rhobar)$, depending only on
$\rhobar$ and $\sigmabar$ (and not on $\tau$ or $\lambda$) such that
for all $\tau,\lambda$, we have
\begin{equation}\label{eqn: BM GL2 intro}
e(R^{\tau,\lambda}/\varpi)=\sum_{\sigmabar}n_{\tau,\lambda}(\sigmabar)\mu_{\sigmabar}(\rhobar).
\end{equation}
Now, the right hand side of (\ref{eqn: BM GL2 intro}) depends only on
$\sigmabar(\tau,\lambda)$, the semisimplification of the reduction modulo
$p$ of $\sigma(\tau,\lambda)$. Let $R_{\Fpbar}(\GL_2(k))$ denote the
Grothendieck group of finite-dimensional $\Fpbar$-representations of
$\GL_2(k)$; then we may define a linear functional
$\iota:R_{\Fpbar}(\GL_2(k))\to\Z$ by sending $\sigmabar$ to
$\mu_{\sigmabar}(\rhobar)$. Then the right hand side of (\ref{eqn: BM
  GL2 intro}) is just $\iota(\sigmabar(\tau,\lambda))$ by definition. 

With this perspective in mind, let $D$ be the non-split quaternion
algebra with centre $K$, and let $\cO_D$ be the maximal order in
$D$. Suppose that $\tau$ is a discrete series type (that is, it is
scalar or it can be extended to an irreducible representation of
$G_K$). As explained in Section~\ref{sec:types-superc-repr}, a natural
analogue of the procedure above associates a finite-dimensional
representation $\sigma_D(\tau,\lambda)$ of $\cO_D^\times$ to the pair
$(\tau,\lambda)$. If $l$ is the quadratic extension of $k$, then
irreducible mod $p$ representations of $\cO_D^\times$ factor through
$l^\times$, so we see that the natural analogue of the
Breuil--M\'ezard conjecture for $D^\times$ is to ask for a linear
functional $\iota_D:R_{\Fpbar}(l^\times)\to\Z$ with the property that
for all pairs $(\tau,\lambda)$ where $\tau$ is discrete series, we
have
\[e(R^{\tau,\lambda,\ds}/\varpi)=\iota_D(\sigmabar_D(\tau,\lambda)),\]
where $R^{\tau,\lambda,\ds}$ denotes the maximal quotient of
$R^{\tau,\lambda}$ corresponding to discrete series lifts (see
Section~\ref{sec: BM conjecture} for more details).

Our approach in this paper is to deduce the existence of $\iota_D$
from the existence of $\iota$. The existence of such functionals for all
representations $\rhobar$ strongly suggests the possibility of there
being a homomorphism $\JL: R_{\Fpbar}(l^\times)\to
R_{\Fpbar}(\GL_2(k))$ such that $\iota_D=\iota\circ\JL$, and the
  construction of such a map 
 is the main objective of this
  paper. Since elements of the Grothendieck group are determined by
  their Brauer characters, this determines a map between the
  class functions on the semisimple conjugacy classes of $\GL_2(k)$
  and $\ltimes$. The usual Jacquet--Langlands correspondence involves a
  sign-reversing relation between the characters evaluated at regular
  elliptic elements; our correspondence satisfies a close analogue of
  this relation.

  Having written down this map, in order to check that $\iota\circ\JL$
  satisfies the properties required of $\iota_D$, the main fact we
  need to check is that $\JL$ takes $\sigmabar_D(\tau,\lambda)$ to
  $\sigmabar(\tau,\lambda)$ when $\tau$ is of supercuspidal
  type. In other words, we need to check that $\JL$ is compatible with the
  usual Jacquet--Langlands correspondence (or rather the induced
  correspondence for types) and reduction modulo $p$. In order to do
  this, we use results of Carayol~\cite{carayol-cuspidal} on the
  construction of supercuspidal representations as well as results of
  Kutzko~\cite{MR882420} on the
  characters of supercuspidal representations and the characters of the types
  they contain. Fred Diamond has pointed out to us that it is
  presumably also possible to verify this directly using the explicit
  formulas in the appendix to~\cite{breuildiamond}. We suspect that
  the approach taken here will extend to give similar results for
  $\GL_n$ (and that the extension should be relatively straightforward
  when $n$ is prime); we intend to return to this question in future
  work. Florian Herzig pointed out to us that our correspondence $\JL$
  is given (up to a sign) by the reduction modulo $p$ of Deligne--Lusztig
  induction from a non-split torus in $\GL_2$ to $\GL_2$. This
  immediately suggests natural analogues of $\JL$ in the case of $\GL_n$.

  We would like to thank Kevin Buzzard for asking whether there was a
  Breuil--M\'ezard conjecture for quaternion algebras. We would also like to
  thank Matthew Emerton, Florian Herzig, Guy Henniart, Mark Kisin, Vytautas
  Pa{\v{s}}k{\=u}nas, and Shaun Stevens for helpful conversations. 

\subsection{Notation} Fix a prime number $p$ and an algebraic closure
$\Qpbar$ of $\Qp$. This determines an algebraic closure
$\Fpbar$ of $\Fp$.

Fix $K$ a finite extension of $\Qp$ with ring of
integers $\cO_K$, uniformiser $\varpi$, and residue field $k$ of
cardinality $q$. Write $G_K$ for
(a choice of) the absolute Galois group of $K$, and $I_K$ for its
inertia subgroup.

Let $D$ be the (unique up
to isomorphism) non-split quaternion algebra with centre $K$, and let
$\cO_D$ be the maximal order in $D$. Fix a uniformiser $\varpi_D$ of $D$. Let $L$ be the quadratic
unramified extension of $K$, so that $D$ splits over $L$. If $l$ is
the residue field of $L$, then $\cO_D/\varpi_D\cong l$. We let $\nu_D$
denote the valuation on $D$ defined by $\nu_D(x)=\nu_K(\Nm(x))$ where
$\Nm$ is the reduced norm on $D$ and $\nu_K$ the valuation on $K$
normalised by $\nu_K(\varpi)=1$. We define $U_D^0 = \OD$ and  if $a\geq 1$ is an integer, we let 
$U_D^a = 1 + \varpi_D^a \CO_D \subset U_D^0$.

We let $\rec_p$ be the local Langlands correspondence of \cite{ht}, so
that if $\pi$ is an irreducible $\Qpbar$-representation of $\GL_n(K)$,
then $\rec_p(\pi)$ is a Weil--Deligne representation of the Weil group
$W_K$ defined over $\Qpbar$. If $R=(r,N)$ is a Weil--Deligne
representation of $W_K$ (so in particular, $r$ is a representation of
$W_K$ with open kernel and $N$ is a nilpotent endomorphism), then by
$R|_{I_K}$ we mean $r|_{I_K}$.

If $E/\Q_p$ is an algebraic extension and $V$ is a continuous
representation of a compact group $G$ on a finite-dimensional
$E$-vector space $V$, then we define a semisimple representation
$\overline{V}$ of $G$ over the residue field of $E$ as follows: since
$G$ is compact, it stabilizes an $\CO_E$-lattice in $V$. Reducing such
a lattice modulo the maximal ideal of $\CO_E$ and semisimplifying
gives the required representation. This representation is independent
of the choice of lattice by the Brauer--Nesbitt theorem.

If $K$ is a $p$-adic field, $W$ is a de Rham
representation of $G_K$ over $E$, and  $\kappa:K \into E$,
then we will write $\HT_\kappa(W)$ for the multiset of Hodge--Tate
weights of $W$ with respect to $\kappa$.  By definition, the multiset $\HT_\kappa(W)$ contains $i$ with multiplicity
$\dim_{E} (W \otimes_{\kappa,F} \widehat{\overline{F}}(i))^{G_F}
$. Let $\Z^2_+ = \{ (a_1,a_2)\in\Z^2 :
a_1\geq a_2\}$, and fix $\lambda \in
(\Z^2_+)^{\Hom_{\Q_p}(K,E)}$. If $W$ is two-dimensional, then we say that $W$ has \emph{Hodge
  type $\lambda$} if for each $\kappa:K\into E$, we have $\HT_\kappa(W)=\{\lambda_{\kappa,1}+1,\lambda_{\kappa,2}\}$. 

If $G$ is a finite group, we let $R_{\Fpbar}(G)$ denote the
Grothendieck group of the category of finitely generated $\Fpbar[G]$-modules.

If $R$ is a commutative ring, we let $R^{\red}$ denote the maximal
reduced quotient of $R$.

\section{A mod $p$ Jacquet--Langlands correspondence for finite
  groups}\label{sec: mod p JL finite groups}We begin by defining an
analogue of the Jacquet--Langlands correspondence for mod $p$
representations of $\GL_2(\cO_K)$ and $\OD$. The irreducible
mod $p$ representations of these two groups are obtained via inflation
from the irreducible mod $p$ representations of $\GL_2(k)$ and
$l^\times$, and our correspondence is actually between the
Grothendieck groups $R_{\Fpbar}(\GL_2(k))$ and
$R_{\Fpbar}(l^\times)$. 
An element of either Grothendieck group is determined by its Brauer
character so we may equivalently describe our map on the level of
Brauer characters. Both descriptions are given below.
Given an element $\sigmabar$ of either $R_{\Fpbar}(\GL_2(k))$
and $R_{\Fpbar}(l^\times)$, we write $\chi_{\sigmabar}$ for its Brauer
character, which we view as being valued in our fixed $\Qpbar$. Recall
that if $G$ is a finite group, the Brauer character of a finite
$\Fpbar[G]$ module is a function on the $p$-regular conjugacy classes
in $G$. For $G=\GL_2(k)$, the $p$-regular conjugacy classes coincide
with the semisimple conjugacy classes; representative elements for these
conjugacy classes are given by the diagonal matrices, and the matrices $i(z)$,
where $z\in l^\times\setminus k^\times$ and $i:l\into M_2(k)$ denotes a choice of embedding of
$k$-algebras. For $G=l^\times$, the
$p$-regular conjugacy classes are just the elements of
$l^\times$.

\begin{defn}
  \label{defn: JL map on class functions}
We define an additive map
$\JL: R_{\Fpbar}(\ltimes) \to R_{\Fpbar}(\GL_2(k)) $ as
follows:
\begin{itemize}
\item if $\psi: k^\times \to \Fpbar^\times$ is a character, then
\[ \JL([\psi\circ \mathbf{N}_{l/k}]) = [\overline{\sp_{\psi}}] -[ \overline{\psi}\circ \det] \]
\item if $\psi : l^\times \to \Fpbar^\times$ is a character which does
  not factor through the norm $\mathbf{N}_{l/k}$, then
\[ \JL([\psi]) = [\overline{\Theta(\psi)}].\]
\end{itemize}
Here the representations $\sp_\psi$ and $\Theta(\psi)$ are as defined
in~\cite[\S 1]{MR2392355}.
\end{defn}

We let $\CGL$ (resp.\ $\Cl$) be the space of $\Qpbar$-valued class functions on the semisimple
conjugacy classes in $\GL_2(k)$ (resp.\ $l^\times$).  Then we have $\CGL =
R_{\Fpbar}(\GL_2(k))\otimes_{\Z}\Qpbar$ (resp.\ $\Cl =
R_{\Fpbar}(l^\times)\otimes_{\Z}\Qpbar$). Note that $\CGL$ and $\Cl$
have natural ring structures, where multiplication corresponds to the tensor
product on $R_{\Fpbar}(\GL_2(k))$ and $R_{\Fpbar}(l^\times)$. We 
 may also describe $\JL:\Cl\to\CGL$ as follows: If $\chi \in \Cl$ then $\JL(\chi)$ is
  defined by the following rule:
  \begin{align*}
    i(z) & \longmapsto  -\chi(z)-\chi(z^q) && \text{\ if\ }
    z\in\ltimes\setminus\ktimes \\
\begin{pmatrix}
  x&0\\0&x
\end{pmatrix} & \longmapsto (q-1)\chi(x) && \text{\ if\ } x\in\ktimes\\
\begin{pmatrix}
  x&0\\0&y
\end{pmatrix} & \longmapsto 0 && \text{\ if\ } x,y\in\ktimes,x\ne y.
  \end{align*}
  That this
definition agrees with the previous one follows immediately from the
table of Brauer characters in~\cite[\S 1]{MR2392355}.

\section{Types and supercuspidal representations}
\label{sec:types-superc-repr}

In this section we will discuss types
for $\GL_2(\cO_K)$ and $\cO_D^\times$, and the inertial local
Langlands and Jacquet--Langlands correspondences. All representations
in this section will be over $\Qpbar$, unless otherwise stated. Recall that an irreducible
admissible smooth representation of $\GL_2(K)$ is either
one-dimensional, a principal series representation, a twist of the
Steinberg representation, or is a supercuspidal representation. If it
is either supercuspidal or a twist of the Steinberg representation, we
say that it is a \emph{discrete series} representation.  There is a
bijection $\JL$ (the Jacquet--Langlands correspondence) from the
irreducible smooth admissible representations of $D^\times$ (which are
necessarily finite-dimensional) to the discrete series representations
of $\GL_2(K)$ (which are necessarily infinite-dimensional). Under this
correspondence, the 1-dimensional representations of $\D$ correspond
to the twists of the Steinberg representation. More precisely, for each character $\psi :
K^\times \to \Qpbar$, we have $\JL(\psi \circ
\Nm) = \Sp_2(\psi |\;\;|^{-1/2})$. (See 
\cite[p.32]{ht} for this formula and the definition of $\Sp_2(*)$.)

\subsection{Types}
\label{sec:types}

In this paper we will be particularly concerned with \emph{types}, which are
finite-dimensional representations of $\GL_2(\cO_K)$ and
$\cO_D^\times$, and with their relationship to inertial types. An
\emph{inertial type} is a two-dimensional representation $\tau$ of
$I_K$ with open kernel which may be extended to a representation of
$G_K$. We say that $\tau$ is a \emph{discrete series} type if it is
either scalar, or can be extended to an irreducible representation of
$G_K$. In the latter case, we say that $\tau$ is \emph{supercuspidal}.

In the $\GL_2$ case, the theory of types is worked out
explicitly in Henniart's appendix to \cite{breuil-mezard}. We recall
his main result. (We follow \cite{kisinfmc}
in introducing the notation $\sigma^{\crs}(\tau)$.)

\begin{thm}\label{thm: Henniart's types} For any inertial type $\tau$, there
  are unique finite dimensional irreducible representations
  $\sigma(\tau)$ and $\sigma^{\crs}(\tau)$ of $\GL_2(\cO_K)$, with the
  following properties:
  \begin{enumerate}
    \item  if $\pi$ is
  an infinite dimensional smooth irreducible representation
  of $\GL_2(K)$, then $\Hom_{\GL_2(\cO_K)}(\sigma(\tau),\pi)\ne 0$ if
  and only if $\rec_p(\pi)|_{I_K}\cong\tau$, in
  which case $\Hom_{\GL_2(\cO_K)}(\sigma(\tau),\pi)$ is
  one-dimensional.
\item  if $\pi$ is any smooth irreducible representation
  of $\GL_2(K)$, then we have $\Hom_{\GL_2(\cO_K)}(\sigma^{\crs}(\tau),\pi)\ne 0$ if
  and only if $\rec_p(\pi)|_{I_K}\cong\tau$ and the monodromy operator
  $N$ on $\rec_p(\pi)$ is 0. In this case,
  $\Hom_{\GL_2(\cO_K)}(\sigma^{\crs}(\tau),\pi)$ is
  one-dimensional.
  \end{enumerate}
 
\end{thm}

There is an analogous (but much simpler) theory for $D^\times$, which
we now recall, following Section 5.2 of \cite{geekisin}.  
Note that
$K^\times\cO_D^\times$ has index two in $D^\times$. Thus if $\pi_D$ is
an admissible smooth representation of $D^\times$, then
$\pi_D|_{\cO_D^\times}$ is either irreducible or a sum of two
irreducible representations which are conjugate under a uniformiser $\varpi_D$ in
$D^\times$. Moreover, we easily see that if $\pi_D'$ is another smooth
irreducible representation of $D^\times$, then $\pi_D$ and $\pi'_D$ differ
by an unramified twist if and only if
$\pi_D|_{\cO_D^\times}\cong\pi'_D|_{\cO_D^\times}$.

Let $\tau$ be a discrete series inertial type. Then by the
Jacquet--Langlands correspondence, there is an irreducible smooth
representation $\pi_{D,\tau}$ of $D^\times$ such that
$\rec_p(\operatorname{JL}(\pi_{D,\tau}))|_{I_K}\cong\tau$. Define
$\sigma_D(\tau)$ to be \emph{one} of the irreducible components of
$\pi_{D,\tau}|_{\cO_D^\times}$; then by the above discussion, we have
the following property.
\begin{thm}\label{thm: types for quaternion algebras}
   Let $\tau$ be a discrete series inertial type. If $\pi_D$ is a smooth irreducible
    $\Qpbar$-representation of $D^\times$ then
    $\Hom_{\cO_D^\times}(\sigma_D(\tau),\pi_D)$ is non-zero if and only
    if
    $\rec_p(\operatorname{JL}(\pi))|_{I_K}\cong\tau$,
    in which case $\Hom_{\cO_D^\times}(\sigma_D(\tau),\pi)$ is
    one-dimensional.
\end{thm}
\begin{rem}
  \label{rem: uniqueness of type in Dtimes case}By the above
  discussion, any representation satisfying the property of
  $\sigma_D(\tau)$ given in Theorem \ref{thm: types for quaternion
    algebras} is necessarily isomorphic to $\sigma_D(\tau)$ or to $\sigma_D(\tau)^{\varpi_D}$.
\end{rem}

For our purposes, we will however require some more precise results:
we will need to know exactly when $\pi_{D,\tau}|_{\OD}$ is irreducible
and we will need to relate the characters of $\sigma(\tau)$ and
$\sigma_D(\tau)$ in a sense we will make precise below.

\subsection{Supercuspidal representations}
\label{sec:superc-repr}

Let $\pi$ be a smooth irreducible representation of $\G$
or $\D$. Then $\pi$ is said to be \emph{minimal} if $\pi$ has minimal
conductor amongst all its twists by characters. Following
\cite{carayol-cuspidal}, we define subgroups
$Z_s$, $K_s$ of $\G$ for $s=1,2$ as follows:
\begin{itemize}
\item $Z_1 = \langle \varpi\rangle$, $K_1 = \GO$,
\item $Z_2 = \langle
  \begin{pmatrix}
    0 & 1 \\ \varpi & 0  \end{pmatrix} \rangle$, $K_2 = \left\{
      \begin{pmatrix}
        a & b \\ c & d
      \end{pmatrix} 
\in \GO : a,d \in \CO^\times, c \in (\varpi), b\in \CO\right\}$. 
\end{itemize}
We refer to \cite[\S 4]{carayol-cuspidal} for the definition of a \emph{very
cuspidal} representation of $Z_sK_s$ $Z_2K_2$ of type $m\geq
1$. When $s=2$, such representations exist only when $m$ is
even.

\begin{thm}[\cite{carayol-cuspidal} Th\'eor\`emes 4.2 \& 8.1]
\label{thm:carayol-classification}

\ 

  \begin{enumerate}
  \item\label{existence} Let $s=1$ or $2$ and set $r=2/s$.  Let $\rho$ be a very
    cuspidal representation of $Z_sK_s$ of type $m$. Then
    $\cind_{Z_sK_s}^{\G}\rho$ is an irreducible minimal
    supercuspidal representation of $\G$ of conductor $r(m-1)+2$.
  \item The representations obtained in \eqref{existence} are all inequivalent.
  \item Every irreducible minimal supercuspidal
    representation of $\G$ is isomorphic to $\cind_{Z_sK_s}^{\G}\rho$
    for a uniquely determined pair $s,\rho$ as in \eqref{existence}.
  \end{enumerate}
\end{thm}

We note that the representations given by \eqref{existence} have even
conductor when $s=1$ and odd conductor when $s=2$. This result in fact
allows us to give an explicit description of the types corresponding to supercuspidal inertial types.

\begin{prop}
  \label{prop:explicit-types-gl2}
Let $\tau : I_K \to \GL_2(\Qpbar)$ be a supercuspidal inertial type. Moreover, assume
that $\tau$ has minimal conductor amongst its twists by smooth
characters that extend to $G_K$. Choose a (necessarily minimal and
supercuspidal) representation $\pi$ of $G$ with $\rec(\pi)$ equal to
an extension of $\tau$ to $W_K$. Write $\pi = \cind_{Z_s K_s}^{\G}\rho$
as in Theorem~\ref{thm:carayol-classification}. Then
\[ \sigma(\tau) \cong \Ind_{K_s}^{\GL_2(\CO_K)} (\rho|_{K_s}).\]
\end{prop}

\begin{proof}
  The construction of $\sigma(\tau)$ described by Henniart in \cite[\S
  A.3.1]{breuil-mezard} is exactly the description given in the
  statement of the proposition.
\end{proof}

When we pass to
representations of $\D$ via the Jacquet--Langlands correspondence,
there is a similar dichotomy which tells us precisely when the
restriction to $\OD$ of a smooth irreducible representation of $\D$ is reducible.

\begin{prop}
\label{prop:types-D}
 Let $\pi_D$ be a smooth irreducible minimal
 representation of $\D$ of dimension greater than 1. Write $c(\pi_D)$ for the
 conductor of $\pi_D$.
 \begin{enumerate}
 \item If $c(\pi_D)$ is odd, then $\pi_D|_{\OD}$ is irreducible.
 \item If $c(\pi_D)$ is even, then $\pi_D|_{\OD} \cong \sigma_D \oplus
   \sigma_D^{\varpi_D}$
for some irreducible representation $\sigma_D$ of $\OD$ with
$\sigma_D\not\cong \sigma_D^{\varpi_D}$.
 \end{enumerate}
\end{prop}

\begin{proof}
  Suppose first of all that $c=c(\pi_D)$ is odd. Let $a = (c-1)/2$ and
  let $\chi$ denote a character of the abelian group
  $U^a_D/U_D^{c-1}$  appearing in $\pi_D|_{U^a_D}$. In \cite[\S
  6.7]{carayol-cuspidal} (where the integer $a$ is denoted $p$), it is
  shown that the stabilizer $Z_{\chi}$ of $\chi$ in $D^\times$ is
  equal to $K(u)^\times U_{D}^a$ where $u \in D$ is an element
  generating a ramified quadratic extension of $K$. In \cite[\S
  6.8]{carayol-cuspidal}, it is shown that
\[ \pi_D \cong \ind_{K(u)^\times U_D^a}^{\D} \psi \]
for some character $\psi$ extending $\chi$.
Now, since $\OD\backslash \D/K(u)^\times U_D^a \cong \Z/\nu_D(K(u)^\times)
=\Z/\Z =0$, it follows that
\[ \pi_D|_{\OD} \cong \ind_{\CO_{K(u)}^\times U_D^a}^{\OD} \psi .\]
Thus, \ $\pi_D|_{\OD}$ is irreducible if and only if for each $t \in \OD-\CO_{K(u)}^\times U_D^a$, the
characters $\psi$ and $\psi^t$ of $H_t:= t\CO_{K(u)}^\times U_D^a
t^{-1} \cap \CO_{K(u)}^\times U_D^a$ are distinct. However, if $\psi =
\psi^t$ on $H_t$ for some $t \in \OD - \CO_{K(u)}^\times U_D^a$, then
since $U^a_D\subset H_t$, we certainly have
$\psi^t|_{U^a_D}=\chi$. Thus, by definition, $t \in Z_\chi =
K(u)^\times U_D^a$, a contradiction.

Suppose now that $c=c(\pi_D)$ is even and set $a = (c-2)/2$. Then in
\cite[\S 6.9]{carayol-cuspidal} it is shown that
\[ \pi_D \cong \ind_{K(u)^\times U_D^a}^{D^\times}\rho \]
where now $u\in D$ generates the quadratic \emph{unramified} extension
of $K$ and $\rho$ is a representation of dimension $1$ or $q^2$. Since 
\[ \OD\backslash \D / K(u)^\times U_D^a \cong \Z/\nu_D(K(u)^\times) =
\Z/2, \]
we deduce immediately that $\pi_D|_{\OD}$ has at least 2
irreducible components. The stated result now follows easily from the
fact that $K^\times \OD$ has index 2 in $\D$.
\end{proof}

We now recall some further results of Carayol.

\begin{prop}
  Let $\pi_D$ be a smooth irreducible minimal
  representation of $\D$ of dimension greater than 1. Let $\pi =
  \JL(\pi_D)$ and write
\[ \pi = \cind_{Z_sK_s}^{\G} \rho \]
for some uniquely determined pair $s,\rho$ as in Theorem
\ref{thm:carayol-classification}~\eqref{existence}. Then
\begin{itemize}
\item If $s=2$, then $(q-1) \dim\pi_D = (q+1) \dim \rho$.
\item If $s=1$, then $(q-1) \dim\pi_D = 2 \dim \rho$.
\end{itemize}
\end{prop}

\begin{proof}
  By \cite[Proposition 7.4]{carayol-cuspidal}, the dimension of
  $\pi_D$ coincides with the formal degree of $\pi$ (when Haar measure
  on $\G/K^\times$ is normalized as in \cite[\S
  5.10]{carayol-cuspidal}). The stated result now follows from the
  formulas obtained in \cite[\S 5.9 -- 5.11]{carayol-cuspidal}.
\end{proof}

We deduce the following formula relating the dimension of types for
$\G$ and $\D$.

\begin{cor}
\label{cor:type-dim-formula}
  Let $\pi_D$ be a smooth irreducible minimal
  representation of $\D$ of dimension greater than 1. Let $\pi =
  \JL(\pi_D)$ and write
\[ \pi = \cind_{Z_sK_s}^{\G} \rho \]
for some uniquely determined pair $s,\rho$ as in Theorem
\ref{thm:carayol-classification}~\eqref{existence}. Define
\[ \sigma : = \ind_{K_s}^{\GO} (\rho|_{K_s}) \]
and let $\sigma_D$ denote an irreducible constituent of
$\pi_D|_{\CO_D^\times}$. Then
\[ (q-1) \dim \sigma_D = \dim \sigma.\]
\end{cor}

\begin{rem}
In our definition of the type $\sigma_D(\tau)$ for $\OD$, we arbitrarily choose one of the
irreducible constituents of $\pi_{D,\tau}|_{\OD}$. This has the apparent
disadvantage of breaking the symmetry of the situation but has the
advantage that the dimension formula above holds independently of the
parity of the conductor. Ultimately in our statement of the
Breuil--M\'ezard conjecture for $D^\times$ we will consider both
choices; see Conjecture~\ref{BM conj for quat alg} and Remark~\ref{rem: uniqueness of functional in quat alg case}.
\end{rem}

Keep the notation of the preceding corollary. We now proceed to show
that the characteristic $p$ reductions $\sigmabar$ and $\sigmabar_D$
of $\sigma$ and $\sigma_D$ are related by the mod $p$
Jacquet--Langlands map defined in Definition~\ref{defn: JL map on class functions}. For this we will make use of results of
Kutzko {\cite{MR882420}}.

We will denote the characters of $\sigma$ and $\sigma_D$ by
$\chi_{\sigma}$ and $\chi_{\sigma_D}$ respectively. Note that the
representation $\sigmabar$ (resp.\ $\sigmabar_D$) factors through the quotient $\GO\onto
\GL_2(k)$ (resp.\ $\OD\onto l^\times$). 
We denote the Brauer character of $\sigmabar$ (resp.\ $\sigmabar_D$) by
$\chi_{\sigmabar} : (\GL_2(k)/\sim)^{\sesi} \to \Zpbar$ (resp.\
$\chi_{\sigmabar_D} : l^\times \to \Zpbar$). Here
$(\GL_2(k)/\sim)^{\sesi}$ is the set of semisimple (or equivalently,
$p$-regular) conjugacy classes in $\GL_2(k)$.

\def\tx{\widetilde{x}}
\def\ty{\widetilde{y}}
\def\tz{\widetilde{z}}
\def\tg{\widetilde{g}}

If $x \in l^\times$, we let $\tx \in \CO_L^\times$ denote its
Teichm\"uller lift. Choose an isomorphism of $\CO_K$-modules $i:
\CO_L \isoto \CO_K\oplus \CO_K$. This gives rise to injections
$i : \CO_L^\times \into \GO$ and $i : l^\times \into
\GL_2(k)$. We also fix an embedding $j:L \into D$ giving rise to
an injection $j: \CO_L^\times \into \OD$.

\begin{prop}
\label{prop:supercusp-char-compat}
    Let $\pi_D$ be a smooth irreducible minimal
  representation of $\D$ of dimension greater than 1. Let $\pi =
  \JL(\pi_D)$ and write
\[ \pi = \cind_{Z_sK_s}^{\G} \rho \]
for some uniquely determined pair $s,\rho$ as in Theorem
\ref{thm:carayol-classification}~\eqref{existence}. Define
\[ \sigma : = \ind_{K_s}^{\GO} \rho|_{K_s} \]
and let $\sigma_D$ denote an irreducible constituent of
$\pi_D|_{\CO_D^\times}$. Then
\[ \JL(\sigmabar_D) = \sigmabar.\]
\end{prop}

\begin{proof} Since both $\JL(\sigmabar_D)$ and $\sigmabar$ are
  semisimple, it suffices to show that we have an equality of Brauer
  characters $\JL(\chi_{\sigmabar_D})=\chi_{\sigmabar}$.
  Let $g \in \GL_2(k)$ be a $p$-regular element. We need to check that $\JL(\chi_{\sigmabar_D})(g) = \chi_{\sigmabar}(g)$.
There are three cases:
  \begin{enumerate}
  \item \emph{We have $g=x$ with $x\in k^\times$.} In this
    case, we need to show that
\[ (q-1)\chi_{\sigmabar_D}(x) = \chi_{\sigmabar}(x) \]
or equivalently, that $(q-1)\chi_{\sigma_D}(\tx) =
\chi_{\sigma}(\tx)$. This follows from Corollary \ref{cor:type-dim-formula} and
the fact that $\pi$ and $\pi_D$ have the same central character.
\item \emph{We have $g\sim\diag(x,y)$ with $x,y\in k^\times$ distinct.}
  In this case we are required to show that
  $\chi_{\sigmabar}(g)=0$, or equivalently, that
  $\chi_{\sigma}(\tg)=0$ where $\tg = \diag(\tx,\ty)$. If $s=2$, this
  follows from \cite[Prop.\ 3.4]{MR882420} and the Frobenius formula for
  the trace of an induced representation. If $s=1$, it follows from
  \cite[Lemmas 6.3 \& 6.4]{MR882420}.
\item \emph{We have $g\sim i(z)$ for some $z \in l^\times \setminus
    k^\times$.} In this case, we need to show that
\[
-\chi_{\sigmabar_D}(z)-\chi_{\sigmabar_D}(z^q)=\chi_{\sigmabar}(i(z)).\]
Let us first consider the sub-case where $s=2$. Then $\pi_D|_{\OD}$ is
irreducible and it suffices for us to show that
\[ -2\chi_{\sigma_D}(j(\tz)) = \chi_{\sigma}(i(\tz)).\]
We will in fact show that both sides vanish. To see that the right
hand side vanishes, recall that $\sigma = \ind_{K_2}^{\GO}\rho$ and
note that for all $t \in \GO$, we have $t^{-1}i(\tz)t \not\in
K_2$. For the left hand side, we have
\[ - \chi_{\sigma_D}(j(\tz)) = -\chi_{\pi_D}(j(\tz)) =
\chi_{\pi}(i(\tz)), \]
where the second equality is a property of the Jacquet-Langlands
correspondence. The vanishing then follows from \cite[Prop.\
5.5(2)]{MR882420}.

Finally, we treat the sub-case where $s=1$. Then $\pi_D|_{\OD}$ is
reducible and it suffices to show that
\[ - \chi_{\sigma_D}(j(\tz)) - \chi_{\sigma_D}(j(\tz^q)) =
\chi_{\sigma}(i(\tz)).\]
For this, note that the left hand side is just $-\chi_{\pi_D}(j(\tz))$,
which in turn equals $\chi_{\pi}(i(\tz))$. Thus we are required to
show that $\chi_{\pi}(i(\tz))=\chi_{\sigma}(i(\tz))$. Yet this follows
from \cite[Prop 6.11(1)]{MR882420} and the proof is complete.
  \end{enumerate}\vspace{-5mm}
\end{proof}

\section{Compatibility of Jacquet--Langlands correspondences}
\label{sec: compatibility of JL with red mod
  p}In this section we prove our main technical result, a
generalization of Proposition~\ref{prop:supercusp-char-compat} which
includes the case of twists of the Steinberg representation (that is,
the case where $\pi_D$ as in
Proposition~\ref{prop:supercusp-char-compat} is one-dimensional) and incorporates algebraic
representations. Again, all representations will be over $\Qpbar$
unless otherwise stated.

Let $W_\lambda$ be an irreducible algebraic representation of
$\Res_{K/\Q_p}\GL_{2_{/K}}$ of highest weight $\lambda$. More precisely, we let $\lambda \in
(\Z^2_+)^{\Hom_{\Q_p}(K,\Qpbar)}$, and take $W_\lambda$ to be the representation 
\[ W_\lambda = \otimes_{\tau \in \Hom_{\Q_p}(K,\Qpbar)}
(\Sym^{a_{\tau,1}-a_{\tau,2}}\otimes\det{}^{a_{\tau,2}})(\Qpbar^2)\]
of $\Res_{K/\Q_p}\GL_{2_{/K}}\times_{\Q_p}\Qpbar =
 \prod_{\tau}\GL_{2_{/\Qpbar}}$. We regard $W_\lambda$ as a representation
of $\GL_2(K)$ via the map $\prod_{\tau} \tau : \GL_2(K) \to \prod_{\tau}\GL_2(\Qpbar)$.
We can also regard it as a
representation of $D^\times$ as follows: choose an isomorphism
$D\otimes_K L \cong M_{2\times 2}(L)$ and for each $\Q_p$-embedding
$\tau : K \into \Qpbar$ choose an embedding $\wt{\tau} : L \into
\Qpbar$ extending $\tau$. Then we regard $W_\lambda$ as representation
of $\D$ via the chain of maps
$D^\times\into
(D\otimes_KL)^\times\cong\GL_2(L) \stackrel{\prod_{\tau}\wt\tau}{\longrightarrow}
\prod_\tau \GL_2(\Qpbar)$. The isomorphism class of the resulting
representation is independent of any choices.
We can then regard $W_\lambda$ as a
representation of $\GL_2(\cO_K)$ or $\cO_D^\times$, by restriction.

Fix a discrete series inertial type $\tau: I_K \to \GL_2(\Qpbar)$, so that we have
finite-dimensional representations $\sigma(\tau)$ and
$\sigma^{\crs}(\tau)$ (resp.\ $\sigma_D(\tau)$)
of $\GL_2(\cO_K)$ (resp.\ $\cO_D^\times$). Define
\begin{align*}
\sigma(\tau,\lambda)&:=\sigma(\tau)\otimes W_\lambda\\
\sigma^{\crs}(\tau,\lambda) &:= \sigma^{\crs}(\tau)\otimes W_\lambda\\
\sigma_D(\tau,\lambda)&:=\sigma_D(\tau)\otimes W_\lambda,
\end{align*}
regarded as representations of $\GO$ or $\OD$ as appropriate.
Since $\GO$ and $\OD$
are compact, we may consider the corresponding semisimple
$\Fpbar$-representations $\sigmabar(\tau,\lambda)$, $\sigmabar^{\crs}(\tau,\lambda)$ and
$\sigmabar_D(\tau,\lambda)$ obtained by reducing a stable lattice and
semisimplifying. These representations factor through the quotients
$\GO\onto\GL_2(k)$ and $\OD \onto l^\times$.
In the
case $\lambda=0$ (when $W_\lambda$ is the trivial representation), we
will write $\sigmabar(\tau)$, $\sigmabar^{\crs}(\tau)$ and
$\sigmabar_D(\tau)$ for $\sigmabar(\tau,0)$,  $\sigmabar^{\crs}(\tau,0)$  and
 $\sigmabar_D(\tau,0)$.

Let $F_\lambda$ (respectively $F^D_\lambda$) be the representation of
$\GL_2(k)$ (respectively $\ltimes$) obtained from
$W_\lambda|_{\GL_2(\CO_K)}$ (resp.\ $W_\lambda|_{\CO_D^\times}$) by
taking a stable lattice, reducing mod $p$, and semisimplifying. The
following lemma is trivial.
\begin{lem}
  \label{comparison of characters of W lambda}We have $\chi_{F_\lambda}\left(
    \begin{pmatrix}
      x&0\\0&x
    \end{pmatrix}\right)=\chi_{F^D_\lambda}(x)$ for each
  $x\in\ktimes$, and $\chi_{F_\lambda}(i(z))=\chi_{F^D_\lambda}(z)$
  for each $z\in\ltimes\setminus\ktimes$.
\end{lem}

The following theorem expresses the compatibility of our mod $p$
Jacquet--Langlands correspondence with the reduction modulo $p$ of the inertial correspondence.

\begin{thm}
  \label{thm:compatibility-JL-redn-mod-p}
Let $\tau: I_K \to \GL_2(\Qpbar)$ be a discrete series inertial type.
\begin{enumerate}
\item Suppose $\tau$ is scalar. Then for each highest weight $\lambda \in
  (\Z^2_+)^{\Hom_{\Q_p}(K,\Qpbar)}$, we have
\[ \JL(\sigmabar_D(\tau,\lambda)) = \sigmabar(\tau,\lambda) -
\sigmabar^{\crs}(\tau,\lambda).\]

\item Suppose $\tau$ is supercuspidal. Then for each highest weight $\lambda \in
  (\Z^2_+)^{\Hom_{\Q_p}(K,\Qpbar)}$, we have
\[ \JL(\sigmabar_D(\tau,\lambda)) = \sigmabar(\tau,\lambda).\]
\end{enumerate}
\end{thm}
\begin{proof}
Since all of the representations involved are semisimple, it suffices to prove
equalities of Brauer characters. By definition we have
$\chi_{\sigmabar(\tau,\lambda)}=\chi_{\sigmabar(\tau)}\chi_{F_\lambda}$, $\chi_{\sigmabar^\crs(\tau,\lambda)}=\chi_{\sigmabar^\crs(\tau)}\chi_{F_\lambda}$, and
$\chi_{\sigmabar_D(\tau,\lambda)}=\chi_{\sigmabar_D(\tau)}\chi_{F^D_\lambda}$,
so by Lemma~\ref{comparison of characters of W lambda} we may immediately reduce
to the case $\lambda=0$. 

In case (1), since everything is compatible with twists by characters we may
reduce to the case that $\tau$ is the trivial type; but then $\sigma_D(\tau)$
and $\sigma^\crs(\tau)$ are the trivial representation, and
$\sigma(\tau)=\sp_1$, and the result is immediate from Definition~\ref{defn: JL
  map on class functions}. In case (2), after twisting we may reduce to the case
that $\sigma_D(\tau)$ extends to a minimal representation of $D^\times$, and the
result is immediate from Proposition~\ref{prop:supercusp-char-compat}.
\end{proof}

\section{The Breuil--M\'ezard conjecture}\label{sec: BM conjecture}In
this section we prove the main theorem of this paper, relating the
Breuil--M\'ezard conjectures for $\GL_2(\cO_K)$ and $\cO_D^\times$. We
begin by recalling the Breuil--M\'ezard conjecture, reformulated in
terms of Grothendieck groups, as in the introduction.

Fix a finite $E/\Qp$ with ring of integers $\cO$, uniformiser $\varpi$
and residue field $\F$, and fix a continuous representation
$\rhobar:G_K\to\GL_2(\F)$. Let $R^{\square}$ be the universal lifting
ring of $\rhobar$ on the category of complete Noetherian local
$\cO$-algebras with residue field $\F$. 
Let $\tau$ be an inertial
type and $\lambda$ a weight as in Section \ref{sec:
  compatibility of JL with red mod p}. Extending $E$ if necessary, we
may assume that $\tau$, $\sigma(\tau)$, $\sigma^{\crs}(\tau)$ and
$\sigma_D(\tau)$ (when $\tau$ is a discrete series type) are all defined over
$E$. Then, there is a quotient
$R^{\tau,\lambda}$ of $R^{\square}$ which is reduced and $p$-torsion
free, and is ``universal'' for liftings which are potentially
semistable of Hodge type $\lambda$ and inertial type
$\tau$. Specifically, we take $R^{\tau,\lambda}$ to be the  image of
the natural map $R^{\square} \to
(R^{\square}[1/p])^{\tau,\lambda,\red}$ where
$(R^{\square}[1/p])^{\tau,\lambda}$ is the quotient of
$R^{\square}[1/p]$ constructed in~\cite[Theorem 2.7.6]{kisindefrings}
(where our $\lambda$ corresponds to Kisin's $\mathbf{v}$). There is also a universal lifting ring
$R^{\tau,\lambda,\crs}$ which is reduced and $p$-torsion free, and is
universal for liftings which are potentially crystalline of Hodge type
$\lambda$ and inertial type $\tau$. In this case, we take
$R^{\tau,\lambda,\crs}$ to be the image
of the map $R^{\square}\to (R^{\square}[1/p])^{\tau,\lambda,\crs}$,
where the latter is ring constructed in~\cite[Cor.\
2.7.7]{kisindefrings}; it is reduced by~\cite[Theorem 3.3.8]{kisindefrings}.
If $R$ is a complete local Noetherian $\cO$-algebra with residue field
$\F$, then we write $e(R/\varpi)$ for the Hilbert--Samuel multiplicity
of $R/\varpi$. 

\begin{conj}\label{BM conj}
  (The Breuil--M\'ezard Conjecture for $\GL_2$.)

 (1) There is a linear
  functional $\iota:R_\F(\GL_2(k))\to\Z$ such that for each
  $\tau,\lambda$ we have
  $\iota(\sigmabar(\tau,\lambda))=e(R^{\tau,\lambda}/\varpi)$.

(2) There is a linear
  functional $\iota_\crs:R_\F(\GL_2(k))\to\Z$ such that for each
  $\tau,\lambda$ we have
  $\iota(\sigmabar^\crs(\tau,\lambda))=e(R^{\tau,\lambda,\crs}/\varpi)$.
\end{conj}

\begin{lem}\label{lem: cris and ss BM conjs agree}If Conjecture~\ref{BM conj} holds, then we necessarily have $\iota=\iota_\crs$.
\end{lem}
\begin{proof}
Since $R^{\tau,\lambda}=R^{\tau,\lambda,\crs}$ and $\sigma^\crs(\tau,\lambda)=\sigma(\tau,\lambda)$ unless $\tau$ is a scalar type,
it is enough to show that $\iota$ (and thus $\iota_\crs$) is uniquely determined
by its values on the $\sigmabar(\tau,\lambda)$ for $\tau$ non-scalar. We may
replace $R_\F(\GL_2(k))$ by $R_\F(\GL_2(k))\otimes_\Z\Qpbar$, so it suffices to
prove that $\cC(\GL_2(k))$ is spanned by the Brauer characters $\chi_{\sigmabar(\tau,\lambda)}$ for
$\tau$ non-scalar. Now,
$\chi_{\sigmabar(\tau,\lambda)}=\chi_{\sigmabar(\tau)}\chi_{F_\lambda}$, and the
$\chi_{F_\lambda}$ span $\cC(\GL_2(k))$, so the span of the $\chi_{\sigmabar(\tau,\lambda)}$ for
$\tau$ non-scalar is an ideal in  $\cC(\GL_2(k))$ (the ideal generated
by the $\chi_{\sigmabar(\tau)}$ for $\tau$ non-scalar). 

Since the maximal ideals in $\cC(\GL_2(k))$ are given by the sets of functions
which vanish on some semisimple conjugacy class, it suffices to show that for
each semisimple conjugacy class, there is some non-scalar type $\tau$ such that
$\chi_{\sigmabar(\tau)}$ does not vanish on that class; but this follows
immediately from the table of Brauer characters in~\cite[\S 1]{MR2392355}.
\end{proof}

The obvious variant for representations of $D^\times$ is as
follows. Let $\lambda$ and $\tau$ be as above. Let $R^{\tau,\lambda,\ds}$
denote the maximal reduced $p$-torsion free quotient of $R^{\tau,\lambda}$
which is supported on the set of irreducible components of $\Spec
R^{\tau,\lambda}$ where the associated Weil--Deligne representation is generically
of discrete series type. More specifically, if $\tau$ is a supercuspidal type, then
$R^{\tau,\lambda,\ds}=R^{\tau,\lambda}=R^{\tau,\lambda,\crs}$; if $\tau$ is a principal
series type, then $R^{\tau,\lambda,\ds}=0$; while if $\tau$ is a
scalar type, then $\Spec R^{\tau,\lambda,\ds}$ is the union the
irreducible components of $\Spec R^{\tau,\lambda}$ not occurring in $\Spec R^{\tau,\lambda,\crs}$, with the reduced
induced scheme structure. Recall that for some $\tau$ of supercuspidal
type, we chose $\sigma_D(\tau)$ to be one of the two irreducible
constituents of the restriction to $\cO_D^\times$ of a certain
supercuspidal representation; in the statement of the following
conjecture, we consider both choices.

\begin{conj}\label{BM conj for quat alg}
  There is a linear functional $\iota_D:R_{\F}(\ltimes)\to\Z$ such that
  for each discrete series type $\tau$, each algebraic weight
  $\lambda$, and each choice of $\sigma_D(\tau)$ we have
  $\iota_D(\chi_{\sigmabar_D(\tau,\lambda)})=e(R^{\tau,\lambda,\ds}/\varpi)$.
\end{conj}

\begin{rem}
  \label{rem: uniqueness of functional in quat alg case}As in the proof of Lemma~\ref{lem: cris and ss BM conjs agree}, a functional $\iota_D$ as in Conjecture \ref{BM conj for quat
    alg} is necessarily unique. Note that in the case that there are
  two choices of $\sigma_D(\tau)$, the two possibilities are related by conjugation
  by $\varpi_D$, and in the other case $\sigma_D(\tau)$ is invariant
  under conjugation by $\varpi_D$. The representation $W_\lambda$ is
  also invariant under conjugation by $\varpi_D$ (as it is a
  representation of $D^\times$). Conjugation by $\varpi_D$ induces
  the involution $c:x\mapsto x^q $ on $l^\times$, so rather than
  insisting on allowing both choices of $\sigma_D(\tau)$ in the
  statement of Conjecture~\ref{BM conj for quat alg}, we could
  equivalently have insisted that $\iota_D$ be invariant under the
  action of $c$, and only used one choice of $\sigma_D(\tau)$.
\end{rem}

Before stating our main result, we note that in the case where $\tau$
is a scalar type, the potentially semistable deformation ring of
weight $\lambda$ and type $\tau$ constructed in~\cite{kisindefrings}
is not necessarily reduced. More specifically, we denote by
$\tR^{\tau,\lambda}$ the image of the map $R^{\square}\to (R^{\square}[1/p])^{\tau,\lambda}$;
it is $p$-torsion free, equidimensional and its generic fibre is
generically reduced (by~\cite[Theorem
3.3.4]{kisindefrings}). The ring $R^{\tau,\lambda}$ is its maximal reduced quotient.
Similarly, we may consider quotients $\tR^{\tau,\lambda,\ds}$ of the
ring $\tR^{\tau,\lambda}$ that are $p$-torsion free and have support
consisting of
the irreducible components generically of discrete series
type. (There need not be a maximal such quotient.)
The ring $R^{\tau,\lambda,\ds}$ is the maximal reduced quotient of any such
$\tR^{\tau,\lambda,\ds}$.  If we work with these potentially larger
rings $\tR^{\tau,\lambda}$ and $\tR^{\tau,\lambda,\ds}$, then the question
arises as to whether the Hilbert Samuel multiplicities of the special
fibres change. The following lemma shows that this is not the
case. 

\begin{lem}
  \label{prop:hs-mult-and-reducedness}
Let $R$ be a complete Noetherian $\CO$-algebra with residue field $\F$.
Suppose that $R$ is $p$-torsion free, equidimensional, and that
$R[1/p]$ is generically reduced. Then
\[ e(R/\varpi) = e(R^{\red}/\varpi) .\]
\end{lem}

\begin{proof}
  Let $I$ denote the kernel of the surjection $R\onto R^{\red}$. Since 
  $R$ is assumed to be $p$-torsion free, $R^{\red}$ is also $p$-torsion free
  and we thus have an exact sequence
\[ 0 \to I/\varpi I \to R/\varpi\to R^{\red}/\varpi \to 0. \]
Thus $e(R/\varpi)=e(R^{\red}/\varpi)+e(I/\varpi I,R/\varpi)$ (notation
as in \cite[\S 1.3]{kisinfmc}) and we are reduced to
showing that $e(I/\varpi I,R/\varpi)=0$. Since $R[1/p]$ is generically reduced,
the localisation $I_{\wp}$ vanishes for every minimal prime $\wp$ of
$R$. Thus the support of $I$ on $R$ is of dimension strictly smaller
than that of $R$. Since $I \subset R \subset R[1/p]$, each minimal
prime in the support of $I$ is $p$-torsion free. It follows that the
support of $I/\varpi I$ is of dimension strictly smaller than that of
$R/\varpi$. Thus $e(I/\varpi I,R/\varpi)=0$, as required. 
\end{proof}

The main result of this paper is the following.

\begin{thm}
  \label{thm: BM for GL_2 implies it for Dtimes} Conjecture \ref{BM
    conj} implies Conjecture \ref{BM conj for quat alg}.
\end{thm}
\begin{proof}
  Assume that Conjecture \ref{BM conj} holds. Define
  $\iota_D:=\iota\circ\JL$. If $\tau$ is supercuspidal, then by Theorem \ref{thm:compatibility-JL-redn-mod-p}, we
  have 
  $\iota_D(\sigmabar_D(\tau,\lambda))=\iota(\sigmabar(\tau,\lambda))=e(R^{\tau,\lambda}/\varpi)=e(R^{\tau,\lambda,\ds}/\varpi)$,
  as required. If $\tau$ is scalar, then we see in the same way using
  Lemma~\ref{lem: cris and ss BM conjs agree} that
  $\iota_D(\sigmabar_D(\tau,\lambda))=\iota(\sigmabar(\tau,\lambda)-\sigmabar^\crs(\tau,\lambda))=e(R^{\tau,\lambda}/\varpi)-e(R^{\tau,\lambda,\crs}/\varpi)=e(R^{\tau,\lambda,\ds}/\varpi)$. (The
  last equality follows from~\cite[Prop.\ 1.3.4]{kisinfmc}, taking $f$
  to be the map $R^{\tau,\lambda} \to R^{\tau,\lambda,\crs}\oplus R^{\tau,\lambda,\ds}$.)
\end{proof}
\begin{cor}
  \label{cor: D BM for Qp}Suppose that $K=\Qp$ and that $p\ge 5$. Then Conjecture~\ref{BM conj for quat alg} holds.
\end{cor}
\begin{proof}
  Under these hypotheses, Conjecture~\ref{BM conj} holds by the main result of~\cite{paskunasBM}.
\end{proof}

\begin{rem}\label{rem: unconditional cases of BM}
    It should also be possible to use the main
  result of \cite{geekisin} to prove that there is a functional
  $\iota$ satisfying the conclusion of Conjecture \ref{BM conj}
  whenever $\lambda=0$ (the only issue being for scalar types, where
  the results of \cite{geekisin} consider only the potentially
  crystalline, rather than potentially semistable representations; but
  when $\lambda=0$, the only representations excluded are ordinary, so
  it should be possible to prove the automorphy lifting theorems
  necessary to use the machinery of \cite{geekisin}). It would then
  follow that a functional $\iota_D$ as in Conjecture \ref{BM conj for
    quat alg} exists if we restrict to the case $\lambda=0$.
\end{rem}

\begin{rem}
  \label{rem: relate to traditional formulation of BM}It may seem to
  the reader that the proof of Theorem \ref{thm: BM for GL_2 implies
    it for Dtimes} is a little too simple, and that we have avoided
  various technical issues, in particular the formulation of the weight part
  of Serre's conjecture for $\rhobar$, which are usually present in
  discussions of the Breuil--M\'ezard conjecture. However, following
  \cite{geekisin}, the weight
  part of Serre's conjecture can be formulated in terms of the Breuil--M\'ezard
  conjecture; namely, the predicted weights for $\rhobar$ are
  precisely the irreducible representations $\sigmabar$ of $\GL_2(k)$
  for which $\imath(\sigmabar)>0$. (Note that if
  Conjecture~\ref{BM conj} is true, then $\imath(\sigmabar)$ is
  positive whenever it is non-zero; this follows from taking $\tau$ to
  be trivial and $W_\lambda$ to be a lift of $\sigmabar$ in the second
  part of the conjecture.)

  The analogous definition could be made for weights of $D^\times$
  (that is, for irreducible representations of $\ltimes$). In fact, if
  we translate the definition of the weight part of Serre's conjecture
  for quaternion algebras made in \cite[Definition 3.4]{MR2822861} to
  this language, it is easy to see that this is \emph{precisely} the
  definition made there.

  More precisely, let $\sigmabar$ be an $\F^\times$-character of
  $l^\times$, and let $\widetilde{\sigma}$ be its Teichm\"uller
  lift. Then the discussion before Definition~3.2 of~\cite{MR2822861}
  shows that $\widetilde{\sigma}=\sigma_D(\tau)$ for some type $\tau$
  (in fact, the tame type corresponding to
  $\widetilde{\sigma}\oplus\widetilde{\sigma}^q$ via local class field
  theory). Taking $\lambda=0$, we see that
  $\imath_D(\sigmabar)=e(R^{\tau,0,\ds}/\varpi)\ge 0$, which is positive if
  and only if $\rhobar$ has a discrete series lift of weight $0$ and
  type $\tau$. This recovers \cite[Definition 3.4]{MR2822861}.
\end{rem}

\begin{rem}
  \label{rem:formula for serre weights on D side}
  Our results in fact give rise to a formula for the predicted
  $D^\times$ weights of $\rhobar$ in terms of the predicted $\GL_2$
  weights of $\rhobar$. Under the perfect pairing
\[ R_{\F}(\GL_2(k))\times R_{\F}(\GL_2(k)) \to \Z \]
which sends two irreducibles $(\sigmabar,\sigmabar')$ to
$\dim_{\F}\Hom_{\GL_2(k)}(\sigmabar,\sigmabar')$, we can identify the
functional $\iota$ with an element
$\sum_{\sigmabar}\mu_{\rhobar}(\sigmabar)\sigmabar$ of
$R_{\F}(\GL_2(k))$. We have a similar pairing
\[ R_{\F}(l^\times)\times R_{\F}(l^\times) \to \Z \]
which allows us to think of $\iota_D$ as an element of
$R_{\F}(l^\times)$. Moreover, we may consider the adjoint $\JL^* : R_{\F}(\GL_2(k)) \to
R_{\F}(l^\times)$ of the map $\JL$ with respect to these pairings.
 Since $\iota_D = \iota\circ \JL$, we see that for
any element $V$ of $R_{\F}(l^\times)$, we have
\[ (\iota_D,V) = (\iota,\JL(V)) = (\JL^*(\iota),V).\]
In other words, $\iota_D = \JL^*(\iota)$. Note that for any
irreducible $\F[\GL_2(k)]$-representation $\sigmabar$, we have
\[ \JL^*(\sigmabar) = \sum_{\xi} m_\xi(\sigmabar) [\xi] +
\sum_{\chi} m_\chi(\sigmabar) [\chi\circ \mathbf{N}_{l/k}] \]
where $\xi$ runs over characters $l^\times \to \F^\times$ not
factoring through $\mathbf{N}_{l/k}$ and $\chi$ runs over characters
$k^\times \to \F^\times$ and where:
\begin{itemize}
\item $m_\xi(\sigmabar)$ is equal to the multiplicity with which $\sigmabar$ appears in
  $\overline{\Theta(\xi)}$ (which is either 0 or 1 by
  \cite[Proposition 1.3]{MR2392355});
\item $m_\chi(\sigmabar)$ is 1 if $\sigmabar = \sp_\chi$;
  it is $-1$ if $\sigmabar=\chi\circ\det$ and it is 0 otherwise.
\end{itemize}
Thus, we have:
\[ \iota_D = \sum_{\xi}\left(\sum_{\sigmabar}
  m_{\xi}(\sigmabar)\mu_{\rhobar}(\sigmabar)\right)[\xi] +
\sum_{\chi}\left(\mu_{\rhobar}(\sp_\chi)-\mu_{\rhobar}(\chi\circ\det)\right)[\chi\circ \mathbf{N}_{l/k}].\]

\end{rem}

\bibliographystyle{amsalpha} 
\bibliography{BMquatalg}
\end{document}